\documentclass[11pt,reqno]{amsart}
\usepackage[text={160mm, 235mm},centering]{geometry}
\usepackage{float}
\usepackage{appendix}
\usepackage{graphicx}
\usepackage{chngpage}
\usepackage{multicol}
\usepackage{subcaption}
\usepackage{setspace}
\usepackage{multirow}
\usepackage[colorlinks=true, citecolor=blue, urlcolor=blue, linkcolor=blue]{hyperref}
\usepackage[all,cmtip]{xy}
\usepackage{amssymb}
\usepackage{tikz-cd}
\usepackage{amsmath}
\usepackage{relsize}
\usepackage{longtable}
\usepackage{pstricks}
\usepackage{color}
\usepackage{ mathrsfs }
\usepackage[all]{xy}
\usepackage{mathabx,epsfig}

\theoremstyle{plain}
\newtheorem{thm}{Theorem}[section]
\newtheorem{theorem}[thm]{Theorem}

\newtheorem{lemma}[thm]{Lemma}
\newtheorem{corollary}[thm]{Corollary}
\newtheorem{proposition}[thm]{Proposition}
\theoremstyle{plain}

\theoremstyle{definition}
\newtheorem{remark}[thm]{Remark}

\newtheorem{definition}[thm]{Definition}

\numberwithin{equation}{section}

\newcommand{\romea}{\uppercase\expandafter{\romannumeral1}}
\newcommand{\romeb}{\uppercase\expandafter{\romannumeral2}}

\newcommand{\Aut}{{\rm Aut}}

\newcommand{\Diag}{{\rm diag}}
\newcommand{\Rank}{{\rm rk}}

\newcommand{\GL}{{\rm GL}}

\newcommand{\SL}{{\rm SL}}

\newcommand{\NS}{{\rm NS}}

\newcommand{\C}{{\mathbb C}}

\renewcommand{\P}{{\mathbb P}}
\newcommand{\Q}{{\mathbb Q}}
\newcommand{\R}{{\mathbb R}}

\newcommand{\Z}{{\mathbb Z}}

\makeatletter
\@namedef{subjclassname@2020}{%
	\textup{2020} Mathematics Subject Classification}
\makeatother

\begin{document}
	
	\title{A smooth quintic surface with 85 lines and Picard number 43}
	
	\author{Xun Yu and Zigang Zhu}
	
	\address{Center for Applied Mathematics and KL-AAGDM, Tianjin University, Weijin Road 92, Tianjin 300072, P.R. China}
	
	\email{xunyu@tju.edu.cn, zhzg0313@tju.edu.cn}
	
	
	\begin{abstract}
		We construct a pencil of complex quintic surfaces with every smooth member containing exactly $85$ lines. We show that one smooth member has Picard number $43$ and by introducing the notion of a $d$-saturated lattice, we prove that its N\'eron--Severi group is generated by the lines and four conics. We also construct smooth quintic surfaces with $79$ lines, including examples with $24$ pairwise skew lines, and a smooth surface of degree $10$ with $356$ lines, including $112$ pairwise skew lines.
	\end{abstract}
	\maketitle
	\setcounter{tocdepth}{1}
	\section{Introduction}
    We work over the complex number field $\C$. Counting lines on smooth surfaces in   $\P^3$ of degree $d\geq3$ is a classical problem in algebraic geometry. It is well-known that every smooth cubic surface contains exactly $27$ lines. A smooth quartic surface with $64$ lines was discovered by Schur \cite{Sch82} in 1882. In 1943, Segre \cite{Seg43} claimed that this number is maximal for smooth quartic surfaces, but his arguments contain a flaw which was detected and corrected by Rams--Sch\"utt \cite{RS15} using elliptic fibrations; see also Degtyarev--Itenberg--Sert\"oz \cite{DIS17} for a different approach based on lattice theory and the Torelli theorem for K3 surfaces.  For $d\ge 5$, the maximal number in question is still unknown (see e.g. \cite{Miy84}, \cite{BS07}, \cite{Miy09}, \cite{Kol15}, \cite{EH16}, \cite{RS20}, \cite{CZ22}, \cite{BR23}). For $d=5$, this number is at most $127$ (\cite[Theorem 1.2]{RS20}). On the other hand, the previous record for the number of lines on a smooth quintic surface was $75$, attained by Fermat quintic and Barth quintic (see \cite{Xie10}, \cite{RS14}). Our first main result establishes a new record of the number of lines on a smooth quintic surface.    
    \begin{theorem}\label{thm:mainA}
    	For every $t\in\C\setminus\{0,1\}$, the quintic surface
    	$X_t\subset\P^3$ defined by
    	$$(x_3-x_4)(x_1^4-3x_2^4)-6(x_3+x_4)x_1^2x_2^2+x_3x_4(x_4^3-tx_3^3)=0$$
    	is smooth and contains exactly $85$ lines.
    \end{theorem}
        
   The proof is done by direct calculations showing that there are precisely $20$ (resp. $64$) lines meeting with (resp. skew to) the line $V(x_3, x_4)\subset X_{t}$.  Let $X$ be a smooth surface in $\P^3$ of degree $d$. The line configurations also lead us to study the N\'eron--Severi group ${\rm NS}(X)$, an important invariant of $X$. The behavior of the Picard number $\rho(X)$ (i.e., the rank of ${\rm NS}(X)$) for $d\ge 5$ is still far from being well understood. For $d=5$, it is unknown whether the upper bound $\rho(X)\le h^{1,1}(X)=45$ by the Hodge number is sharp (see \cite[Page 288]{Per82}, \cite[Page 112]{Bea14}). To the best of our knowledge, the largest previously known Picard number for a smooth quintic surface was $41$, attained by the Barth quintic surface (\cite{RS14}). One member of the above pencil establishes a new record for the Picard number of a smooth quintic surface.

  \begin{theorem}\label{thm:main-ns}
    	The smooth quintic surface $X_{-1/8}$ in Theorem \ref{thm:mainA} has Picard number $43$. 
    \end{theorem} 
   The intersection matrix of the lines on $X_{-1/8}$ has rank $39$. By searching for low degree curves, we find four explicit smooth conics whose classes together with those of the lines generate a sublattice $M$ in ${\rm NS}(X_{-1/8})$ of rank $43$, which gives a lower bound $\rho(X_{-1/8})\ge 43$. Inspired by \cite{RS14}, we prove the upper bound $\rho(X_{-1/8})\le 43$ by looking at a suitable finite quotient of $X_{-1/8}$ and controlling the Picard number of the minimal resolution via good reduction modulo $p$.   More precisely, we find a quotient of $X_{-1/8}$ by a cyclic group of order $4$ and its minimal resolution $Y$ is a relatively minimal elliptic surface. By a result of Shioda \cite[Proposition 5]{Shi86}, one gets an upper bound for $\rho(X_{-1/8})$ in terms of invariants of $Y$ (see \eqref{eq:boundrho}). Combining this with an upper bound for $\rho(Y)$ from good reduction of $Y$ at $p=13$, we get the upper bound $\rho(X_{-1/8})\le 43$.
      
   It is important to understand the full structure of the N\'eron--Severi group. Once the Picard number is known, a natural question is to find explicit generators of the N\'eron--Severi group. This in general requires nontrivial work, even for smooth surfaces in $\P^3$ (see e.g., \cite{AS83}, \cite{BS08}, \cite{SSvL10}, \cite{RS14}, \cite{Deg15}). Based on the geometry of $X$ and the arithmetic nature of ${\rm NS}(X)$ (Propositions \ref{prop:voisin-bound} and \ref{prop:linediff}), we introduce the notion of a {\it $d$-saturated lattice} (Definition \ref{def:d-saturated}) and develop a criterion for a finite-index sublattice of ${\rm NS}(X)$ containing the classes of all lines on $X$ to coincide with ${\rm NS}(X)$ (Corollary \ref{cor:ns-generation}). Using this new criterion, we obtain the following 
 \begin{theorem}\label{thm:main-basis}
    	The N\'eron--Severi group ${\rm NS}(X_{-1/8})$ is generated by the classes of the lines and four conics.
    \end{theorem}    
In fact, we show that the orthogonal complement $N$ of the hyperplane class $h$ in $M$ is $5$-saturated, which implies $M={\rm NS}(X_{-1/8})$ by Corollary \ref{cor:ns-generation}. The criterion is applicable to other surfaces. For example, Rams--Sch\"utt found a sublattice of finite index in the N\'eron--Severi group of the Barth quintic surface and proved that this index is $2^i$ for some $0\leq i\leq4$ (\cite[Proposition~5.2]{RS14}). Applying Corollary \ref{cor:ns-generation}, one can conclude that this index is $1$ (Remark \ref{rmk:Barth}).

We found the pencil in Theorem \ref{thm:mainA} by considering lines on smooth quintic surfaces $X$ with a large automorphism group ${\rm Aut}(X)$ (see Remark \ref{rem:aut}). Following the same approach, we found further examples of surfaces in $\P^3$ with many lines. Miyaoka \cite[Section~2.2]{Miy84} proved that a smooth surface of degree $d\geq 4$ in $\P^3$ contains at most $2d(d-2)$ pairwise skew lines. For $d=5$ (resp. $10$), this gives an upper bound of $30$ (resp. $160$). To the best of our knowledge, the previous records for the number of pairwise skew lines in degrees $5$ and $10$ were $19$ and $84$, respectively, while the previous record for the total number of lines in degree $10$ was $300$ (see \cite{Ram02}, \cite{Ram05}, \cite{BS07}, \cite{FLR19}).

We construct examples of smooth quintic surfaces with exactly $79$ lines, including $24$ pairwise skew lines (Theorem \ref{thm:79-members}). In degree $10$, we also construct a smooth surface with exactly $356$ lines, including $112$ pairwise skew lines (Theorem~\ref{thm:356-lines}).

 	\subsection*{Acknowledgements}
	This work is partially supported by the National Natural Science Foundation of China (No. 12071337).
	\section{$d$-saturated lattices}
	In this section, we introduce the notion of a $d$-saturated lattice, which is motivated by certain properties of the divisor classes orthogonal to the hyperplane class on smooth surfaces in $\P^3$ (Propositions \ref{prop:voisin-bound} and \ref{prop:linediff}). Then we develop a criterion for a finite-index sublattice of ${\rm NS}(X)$ containing the classes of all lines on $X$ to coincide with ${\rm NS}(X)$ (Corollary \ref{cor:ns-generation}).
	
	We recall some basics on lattices which we will use in our paper  (see e.g. \cite{Nik80} for more details).
	A lattice $(S,(*,**))$ is a free $\Z$-module of finite rank endowed with a $\Z$-valued symmetric bilinear form $(*,**) = (*,**)_S$. We write $x^2=(x,x)$ and $S_k=S\otimes_{\Z}k$ for a field $k$. The lattice $S$ is {\it even} if $x^2\in2\Z$ for every $x\in S$, and {\it positive definite} (resp. {\it negative definite}) if the induced form on $S_{\R}$ is positive definite (resp. negative definite).
	
	Let $(v_1,\ldots,v_m)$ be a $\Z$-basis of $S$. The {\it Gram matrix} of $S$ with respect to $(v_1,\ldots,v_m)$ is $((v_i,v_j))_{1\le i,j\le m}$. The determinant of this matrix is independent of the chosen basis and is denoted by $\det S$. We call $S$ {\it nondegenerate} if $\det S\ne0$, and define its \emph{discriminant} by $\operatorname{disc}S:=|\det S|$.
	
	Let $S$ be a nondegenerate lattice. A sublattice $S'\subseteq S$ is {\it primitive} if $S'_{\Q}\cap S=S'$, equivalently, if $S/S'$ is torsion-free. If $S$ is a sublattice of a lattice $\bar S$ of the same rank, then we say $\bar S$ is an {\it overlattice} of $S$. Note that the index $[\bar S: S]$ is equal to the order of the finite abelian group $\bar{S}/S$, and $[\bar S:S]^2=\operatorname{disc}S/\operatorname{disc}\bar S$. We may view the dual $S^\vee:={\rm Hom}(S,\Z)$ of $S$ as a subset of $S_\Q$ via the canonical embedding $S\hookrightarrow S^\vee$ determined by the bilinear form of $S$. The {\it discriminant group} of $S$ is the finite abelian group $S^\vee/S$ of order $\operatorname{disc}S$. If $S$ is even, its {\it discriminant form} is
	$$q_S:S^\vee/S\longrightarrow\Q/2\Z, \qquad q_S(x+S)\equiv x^2\bmod2\Z.$$
	A subgroup $H\subseteq S^\vee/S$ is {\it isotropic} if $q_S|_H=0$. In this case, $S_H$ denotes the inverse
image of $H$ in $S^\vee$. For an element $x\in S$, we use $x^\perp_S$ to denote the orthogonal complement of $x$ in $S$. For a negative definite lattice $M$ and a positive integer $b$, we define
$$n_b(M):=\#\{v\in M\mid 0<-v^2\leq b\}.$$

\medskip
	
	Let $X\subset\P^3$ be a smooth surface of degree $d$. We denote by $h$ the hyperplane class of $X$. We often use $h_{\NS}^\perp$ to denote $h_{{\NS}(X)}^\perp$ if there is no confusion. The {\it line lattice} $S_X$ of $X$ is the sublattice
	$$S_X:=\bigl\langle [L]\mid L\subset X{\rm\; is \; a \; line}\bigr\rangle\subseteq\NS(X)$$
 generated by the classes of all lines on $X$.
	\begin{definition}\label{def:d-saturated}
	    Let $d\geq 5$ be a positive integer and let $N$ be an even negative definite lattice. We say that $N$ is {\it $d$-saturated} if every nontrivial even overlattice $N\subsetneq{\bar N}$ contains a vector $x\in{\bar N}\setminus N$ satisfying $x^2\geq -2(d-2)$.
	\end{definition}
	This definition is motivated by Propositions \ref{prop:voisin-bound} and \ref{prop:linediff}.

    The lattice $h_{\NS}^\perp$ is even and negative definite by
	Riemann--Roch and the Hodge index theorem.
	Feyzbakhsh--Thomas proved the following bound for $d\geq10$
	using wall-crossing \cite[Theorem~1.1(B)]{FT21}.
	In the appendix to their paper, Voisin proved
	$D^2\leq-2d+5$ for smooth surfaces in $\P^3$ of degree $d\geq4$
	\cite[Theorem~A.1]{FT21}; evenness gives the formulation below.
	
	\begin{proposition}[Feyzbakhsh--Thomas, Voisin]
		\label{prop:voisin-bound}
		Let $X\subset\P^3$ be a smooth surface of degree $d\geq4$,
		with hyperplane class $h$. Every nonzero
		$D\in h_{\NS}^\perp$ satisfies $D^2\leq-2(d-2)$.
	\end{proposition}
	
	Feyzbakhsh--Thomas also observed that differences of classes
	of skew lines attain this bound \cite[Section~1]{FT21}.
	For quintic surfaces, the converse follows from \cite[Lemma~2.1]{Ram02}. We give a uniform proof for all $d\geq5$.
	
	\begin{proposition}\label{prop:linediff}
		Let $X\subset\P^3$ be a smooth surface of degree $d\geq5$,
		with hyperplane class $h$. A class $D\in h_{\NS}^\perp$
		satisfies $D^2=-2(d-2)$ if and only if
		$D=[L_1]-[L_2]$ for two skew lines $L_1,L_2\subset X$.
		The ordered pair $(L_1,L_2)$ is uniquely determined by $D$.
	\end{proposition}
	
	\begin{proof}
		Let $D\in h_{\NS}^\perp$ with $D^2=-2(d-2)$. Put $p_g=\binom{d-1}{3}$ and
		$f_i=h^0(X,\mathcal O_X(ih-D))$. Ampleness gives $f_0=0$.
		The restriction estimates in Voisin's proof
		\cite[Proof of Theorem~A.1]{FT21}, applied to a very general
		plane section, give
		$$
		f_i-f_{i-1}\leq\binom{i+2}{2}-1,
		\qquad 1\leq i\leq d-4.
		$$
		Riemann--Roch and Serre duality therefore imply
		$$
		p_g-d+3\leq f_{d-4}
		\leq\sum_{i=1}^{d-4}\left(\binom{i+2}{2}-1\right)
		=p_g-d+3.
		$$
		All the increment inequalities are equalities, so $f_1=2$.
		
		Write $|h-D|=B+|P|$, with fixed divisor $B$ and moving
		pencil $|P|$. Since $(h-D)^2=4-d<0$ and $P^2\geq0$,
		we have $B\ne0$ and $(P,h)\leq d-1$.
		An integral component $T$ of a general member of $|P|$
		moves in an algebraic family, so $[T]^2\geq0$.
		Writing $e=\deg T$, adjunction and the degree--genus bound $p_a(T)\leq (e-1)(e-2)/2$ (see \cite[Theorem~3.1 and Proposition~3.4]{GP78}) give
		$$
		0\leq[T]^2=2p_a(T)-2-(d-4)e\leq e(e+1-d).
		$$
		Thus $e\geq d-1$. It follows that a general member of $|P|$
		is reduced and integral of degree $d-1$, with $P^2=0$ and
		genus $(d-2)(d-3)/2$.
		The pencil is base-point-free, so its general member is smooth
		by Bertini and planar by Castelnuovo's bound
		\cite[IV, Theorem~6.4]{Har77}.
		Its residual divisor in the plane section is a line $L_1$,
		and $B$ is a line $L_2$. Hence
		$h-D=[L_2]+h-[L_1]$, giving $D=[L_1]-[L_2]$.
		Since $D^2=-2(d-2)$ and $[L_i]^2=2-d$, the lines are skew.
		
		Finally, $(D,[L_1])=2-d$ and $(D,[L_2])=d-2$, whereas
		$(D,[L])\in\{-1,0,1\}$ for every other line $L$.
		Since $d\geq5$, these values determine the ordered pair
		uniquely.
\end{proof}
\begin{remark}
For $d=3$, the ordered pair of skew lines is not necessarily uniquely determined by $D$. For $d=4$, the class $D$ need not be a difference
 of line classes. For example, the Fermat quartic $X=V(\sum_{j=1}^4x_j^4)$ contains the smooth elliptic quartic
 $$E=V(x_1^2+x_2^2-i\sqrt2x_3^2,\,x_1^2-x_2^2-i\sqrt2x_4^2).$$
 The class $D=[E]-h$ satisfies $D\cdot h=0$ and $D^2=-4$.
 If $D=[L_1]-[L_2]$ for two skew lines, then
 $[E]\cdot[L_1]=(h+[L_1]-[L_2])\cdot[L_1]=1-2=-1$,
 which is impossible since $E$ and $L_1$ are distinct
 irreducible curves.
\end{remark}

	\begin{corollary}\label{cor:ns-generation}
		Let $X\subset\P^3$ be a smooth surface of degree $d\geq 5$, with hyperplane class $h$.  Suppose that $X$ contains at least one line. Let $C_1,\ldots,C_r\subset X$ be curves, and let
		$$M:=\bigl\langle [C_1],\dots,[C_r]\bigr\rangle\subseteq {\rm NS}(X),\qquad N:=M\cap h_{\NS}^\perp.$$
		Suppose that the line lattice $S_X\subseteq M$ and $\Rank\, M=\rho(X)$. If $N$ is $d$-saturated, then $M=\NS(X)$.
	\end{corollary}
	\begin{proof}
		Since $\Rank\, M=\rho(X)$, the lattice $h_{\NS}^\perp$ is an even overlattice of $N$. Since $N$ is $d$-saturated, if $N\subsetneq h_{\NS}^\perp$, there exists $D\in h_{\NS}^\perp\setminus N$ satisfying $D^2\geq -2(d-2)$. By Propositions \ref{prop:voisin-bound} and \ref{prop:linediff}, equality holds and $D=[L_1]-[L_2]$ for two skew lines $L_1,L_2\subset X$. Since $S_X\subseteq M$, we have $D\in N$, a contradiction. Hence $N=h_{\NS}^\perp$.
		
		Choose a line $L\subset X$. Since $[L]\in M$ and $(h, [L])=1$, every $D'\in\NS(X)$ satisfies $$D'-(D',h)[L]\in h_{\NS}^\perp=N\subseteq M.$$ Thus $D'\in M$, and $M=\NS(X)$.
	\end{proof}
	
\begin{lemma}\label{lem:saturation-criterion}
	Let $d\geq5$ be an integer and $N$ an even negative definite
	lattice. 
	Then $N$ is $d$-saturated if and only if
	$$
	n_{2(d-2)}(N_H)>n_{2(d-2)}(N)
	$$
	for every isotropic subgroup $H\subseteq N^\vee/N$ of prime
	order $p$ satisfying $p^2\mid\operatorname{disc}N$.
\end{lemma}

\begin{proof}
	By \cite[Proposition~1.4.1]{Nik80}, even overlattices of $N$ correspond bijectively to isotropic subgroups $H'\subseteq N^\vee/N$, with $[N_{H'}:N]=|H'|$.
	Every nontrivial even overlattice contains an intermediate even overlattice of prime index $p$ satisfying $p^2 |\operatorname{disc}N$.
	Therefore, it suffices to consider the subgroups in the lemma.
	Since $N\subseteq N_H$, the inequality holds if and only if $N_H\setminus N$ has a vector of square at least $-2(d-2)$. 
\end{proof}

In our applications, the degree $d=5$, and we often use the function $\texttt{qfminim}$ in PARI/GP (\cite{Th}) to determine $n_6(N)$.

	\section{A pencil of quintic surfaces with 85 lines}
	In this section, we prove Theorem~\ref{thm:mainA} by determining all lines on the smooth members of the pencil $X_t$.
	We also give explicit equations for the $85$ lines on $X_{-1/8}$.
	
	\subsection{The pencil and its smooth members}
	Consider the quintic surfaces $X_t\subset \P^3$ defined by
	\begin{equation}\label{eq:85-pencil}
			F_t=(x_3-x_4)(x_1^4-3x_2^4)-6(x_3+x_4)x_1^2x_2^2+x_3x_4(x_4^3-tx_3^3),
	\end{equation}
	where $t\in\C$.
		\begin{lemma}\label{lem:85-smoothness}
		The surface $X_t$ is smooth if and only if $t\notin\{0,1\}$.
	\end{lemma}
	\begin{proof}
		For $t=0,1$, clearly $X_t$ is singular. From now on, we may assume that $t\neq 0,1$. By computing the partial derivatives $\frac{\partial F_{t}}{\partial{x_i}}$, we infer that the lines $L_1:=V(x_3,x_4)$ and $V(x_1,x_2)$ in $\mathbb{P}^3$ contain no singular points of $X_t$.
		
		Now, suppose $(x_1:x_2:x_3:x_4)$ is a singular point of $X_t$ with at least one of $x_3$ and $x_4$ being nonzero. If exactly one of $x_1, x_2$ is zero, then the partial derivative with respect to the nonzero coordinate being zero gives $x_3=x_4$, which implies that $x_3 x_4\neq 0$. This together with $F_t=0$ gives $t=1$, a contradiction. If $x_1x_2\neq0$, the equations $\frac{\partial F_{t}}{\partial{x_i}}=0$ ($i=1,2$) give $(x_3-x_4)x_1^2=3(x_3+x_4)x_2^2$ and $(x_3+x_4)x_1^2=-(x_3-x_4)x_2^2$. Thus $x_1^4=-3x_2^4$, $x_3^2+x_3x_4+x_4^2=0$. Using Euler's identity, we have 
$$0=x_1 \frac{\partial F_{t}}{\partial{x_1}}+x_2 \frac{\partial F_{t}}{\partial{x_2}}=4(F_t- x_3x_4(x_4^3-tx_3^3)),$$ which implies $x_3x_4(x_4^3-tx_3^3)=0$ since $F_t=0$. By $(x_3, x_4)\neq (0,0)$ and $x_3^2+x_3x_4+x_4^2=0$, we infer that $t=1$, a contradiction.
	\end{proof}
	\subsection{The 85 lines}
	The line $L_1:=V(x_3,x_4)\subset \mathbb{P}^3$ is contained in the quintic surface $X_t$, where $t\in \mathbb{C}\setminus\{0,1\}$. Consider 
	$$\mathcal{I}_t:=\{L\,|\, L\neq L_1 \text{ is a line on }X_t \text{ meeting }L_1\}$$	and 
$$\mathcal{S}_t:=\{L\,|\, L \text{ is a line on }X_t \text{ skew to }L_1\}.$$
	\begin{proof}[Proof of Theorem \ref{thm:mainA}]
		 Smoothness of $X_t$ follows from Lemma \ref{lem:85-smoothness}.
		 Each line in $\mathcal{I}_t$ is contained in exactly one of the planes $P_{(s_1:s_2)}=V(s_1 x_3-s_2 x_4)\subset \mathbb{P}^3$, where $(s_1:s_2)\in \mathbb{P}^1$. Clearly the intersection $P_{(1:0)}\cap X_{t}=V(x_3, x_4(x_1^4-3 x_2^4+6 x_1^2 x_2^2))\subset \mathbb{P}^3$ consists of $5$ distinct lines, which gives $4$ members in $\mathcal{I}_t$. Let $\lambda\in\mathbb{C}$. The plane $P_{(\lambda:1)}$ contains a line in  $\mathcal{I}_t$ if and only if the residual equation 
$$(1-\lambda)(x_1^4-3x_2^4)-6(1+\lambda)x_1^2 x_2^2+\lambda (\lambda^3-t)x_3^4=0$$
has a linear factor.	If $\lambda (\lambda^3-t)\neq 0$, the existence of such a factor implies  $$(1-\lambda)(x_1^4-3x_2^4)-6(1+\lambda)x_1^2 x_2^2=(a x_1 +b x_2)^4$$ for some $a,b\in \mathbb{C}$, which is impossible. If $\lambda (\lambda^3-t)=0$, then $P_{(\lambda:1)}\cap X_{t}$ contains $5$ distinct lines. This gives $4\cdot 4=16$ lines in $\mathcal{I}_t$. Thus totally $\mathcal{I}_t$ has $20$ lines.

For a matrix $A=\begin{pmatrix}
a_{11} & a_{12} \\
a_{21} & a_{22}
\end{pmatrix}
$, we define a line $$L_A:=\{(a_{11}x_3+a_{12}x_4:a_{21}x_3+a_{22}x_4:x_3:x_4)|\, (x_3:x_4)\in \mathbb{P}^1\}\subset \mathbb{P}^3.$$
Note that this gives a bijection between the set ${\rm M}(2,\mathbb{C})$ of $2$ by $2$ matrices and the lines in $\mathbb{P}^3$ skew to $L_1$. Suppose $L_A\in \mathcal{S}_t$. Then the polynomial $$F_{t, A}:=F_t(a_{11}x_3+a_{12}x_4,a_{21}x_3+a_{22}x_4,x_3, x_4)\in \mathbb{C}[x_3, x_4]$$ is zero. Then the coefficients of the monomials $x_3^{5-i} x_4^{i}$ ($0\le i\le 5$) in this polynomial must be zero.
 From this and $t\neq 0$, we get $a_{11} a_{12}a_{21}a_{22}\neq 0$. To simplify computation, we may assume $A=\lambda \begin{pmatrix}
r u & v \\
u & 1
\end{pmatrix}
$, where $\lambda, r, u, v\in \mathbb{C}^*$. Then by direct computation, the conditions of the coefficients of $x_3^{5-i} x_4^{i}$ ($0\le i\le 5$) in $F_{t, A}$ being zero are precisely the equations 
\begin{equation}\label{eq1}
\lambda^4 u^4 (-3 - 6 r^2 + r^4)=0,
\end{equation}
\begin{equation}\label{eq2}
-t - \lambda^4 u^3 (12 + 12 r^2 - 3 u + 6 r^2 u + r^4 u + 12 r v - 4 r^3 v)=0,
\end{equation}
\begin{equation}\label{eq3}
-2 \lambda^4 u^2 (9 + 3 r^2 - 6 u + 6 r^2 u + 12 r v + 6 r u v + 2 r^3 u v + 3 v^2 - 
 3 r^2 v^2)=0,
\end{equation}
\begin{equation}\label{eq4}
-2 \lambda^4 u (6 - 9 u + 3 r^2 u + 6 r v + 12 r u v + 6 v^2 + 3 u v^2 + 
 3 r^2 u v^2 - 2 r v^3)=0,
\end{equation}
\begin{equation}\label{eq5}
1 + \lambda^4 (-3 + 12 u - 12 r u v - 6 v^2 - 12 u v^2 - 4 r u v^3 + v^4)=0,
\end{equation}
\begin{equation}\label{eq6}
-\lambda^4 (-3 + 6 v^2 + v^4)=0.
\end{equation}

The two equations \eqref{eq1}, \eqref{eq6} give $16$ pairs of $(r, v)$, and exactly $4$ pairs among them lead non-empty solution for the system of the two equations \eqref{eq3} and \eqref{eq4} (note that the long factors in \eqref{eq3}, \eqref{eq4} have degree one in variable $u$). These $4$ pairs of $(r_j,v_j)$ ($1\le j\le 4$) are  
$$( a, -b),\, (-a, b),\, ( b i , -a i ), \, ( -b i , a i ),$$
where $a=\sqrt{3+2\sqrt{3}}, b=\sqrt{-3+2\sqrt{3}}$.
On the other hand, by direct substitution using $(r_j,v_j)$, the equations \eqref{eq1}, \eqref{eq3}, \eqref{eq4}, \eqref{eq6} become identically zero, and \eqref{eq2}, \eqref{eq5} become $$-t - 12 \lambda^4 u^3 (6 + 2 \sqrt{3} + 3 u + 2 \sqrt{3} u)=0,\,\, 1 + 
 12 \lambda^4 (3 - 2 \sqrt{3} +6 u -2 \sqrt{3}u)=0, \text{ for }j=1,2,$$
 $$-t + 12 \lambda^4 u^3 (-6 + 2 \sqrt{3} - 3 u + 2 \sqrt{3} u)=0,\,\, 
 1 + 12 \lambda^4 (3 + 2 \sqrt{3} + 6u + 2\sqrt{3} u)=0, \text{ for }j=3,4.$$

Thus, for each pair $(r_j,v_j)$, clearly there are exactly $16$ solutions for $(\lambda, u)$. This implies that $\mathcal{S}_t$ contains exactly $4\cdot 16=64$ lines. Therefore, the number of the lines on $X_t$ is $1+20+64=85$. \end{proof}

We label the $85$ lines on $X_{-1/8}$ as follows, starting with $L_1=V(x_3,x_4)$. Throughout, $j_1,j_2,j_3\in\{0,1\}$, $k\in\{0,1,2,3\}$, and $i=\sqrt{-1}$.
The $20$ lines meeting $L_1$ are
\begingroup
\setlength{\jot}{3pt}
$$
\begin{aligned}
	L_{2+4j_1+2j_2+j_3}:\quad&
	x_{4-j_1}=0,\qquad
	\bigl(\sqrt3-(-1)^{j_1+j_2}\bigr)x_1
	=(-1)^{j_1+j_3}i^{j_2}\sqrt[4]{12}\,x_2;
	\\[3pt]
	L_{10+2j_2+j_1}:\quad&
	2x_4=-x_3,\qquad
	x_1=(-1)^{j_1}i(i\sqrt3)^{j_2}x_2;
	\\[3pt]
	L_{14+2j_1+j_3}:\quad&
	4x_4=(1-i\sqrt3)x_3,\\
	&\bigl(\sqrt[4]{12}+(-1)^{j_3}(1-i)\bigr)x_1
	=(-1)^{j_1}
	\bigl((-1)^{j_3}\sqrt[4]{12}-(1+i)\sqrt3\bigr)x_2;
	\\[3pt]
	L_{18+2j_1+j_3}:\quad&
	4x_4=(1+i\sqrt3)x_3,\\
	&\bigl(\sqrt[4]{12}-(-1)^{j_1+j_3}(1+i)\bigr)x_1
	=-\bigl((-1)^{j_3}\sqrt[4]{12}
	+(-1)^{j_1}(1-i)\sqrt3\bigr)x_2.
\end{aligned}
$$
The $64$ lines skew to $L_1$ form two families of $32$ lines:
$$
\begin{aligned}
	&L_{22+16j_1+40j_2+4j_3+k}:\\[1pt]
	&\quad\left\{
	\begin{aligned}
		&2\sqrt[4]{18}
		\bigl((-1)^{j_1}x_1-(-1)^{j_2+j_3}ix_2\bigr)
		=i^{k+(1-j_2)(1-j_3)}(i-1)(x_3+2x_4),\\[3pt]
		&4\bigl((-1)^{j_1}(\sqrt3+(-1)^{j_2})x_1
		+i^{j_2}\sqrt[4]{12}\,x_2\bigr)\\
		&\qquad =
		i^{k+(1-j_2)(1-j_3)}(i-1)\sqrt[4]{6}
		\bigl((-1)^{j_2}\sqrt[4]{12}
		+(-1)^{j_3}i^{1-j_2}(\sqrt3-(-1)^{j_2})\bigr)x_3;
	\end{aligned}
	\right.
	\\[7pt]
	&L_{30+16j_1+24j_2+4j_3+k}:\\[1pt]
	&\quad\left\{
	\begin{aligned}
		&2(-1)^{j_1+j_3}\sqrt3\,x_1+6(-1)^{j_2}x_2
		=i^{k+j_2(1-j_3)}\sqrt[4]{6}\,(x_3+2x_4),\\[3pt]
		&(-1)^{j_1}
		\bigl((-i)^{j_2}\sqrt[4]{12}
		+(-1)^{j_3}(1+(-1)^{j_2}\sqrt3)\bigr)x_1\\
		&\qquad+
		\bigl((-1)^{j_3}(-i)^{j_2}\sqrt[4]{12}
		+3-(-1)^{j_2}\sqrt3\bigr)x_2
		=i^{k+j_2(3-j_3)}\sqrt[4]{6}\,x_3.
	\end{aligned}
	\right.
\end{aligned}
$$
\endgroup

By computation with the help of computer, the intersection matrix of the $85$ lines has rank $39$, which implies that  the line lattice $S_{X_{-1/8}}$ has rank $39$.
	
	\begin{remark}\label{rem:aut}
	We denote the primitive $k$-th root $e^{2\pi i/k}$ of unity by $\xi_k$. Consider the following 3 matrices in $\GL(4,\C)$
\begin{equation*}\label{eq:85-aut-genes}
	A_1=\begin{pmatrix}
		-\xi_8/\sqrt2&3^{1/4}/\sqrt2&0&0\\
		-1/(3^{1/4}\sqrt2)&-\xi_8^{7}/\sqrt2&0&0\\
		0&0&\xi_3&0\\
		0&0&0&\xi_3^2
	\end{pmatrix},\; 
	A_2=\begin{pmatrix}
		-1&0&0&0\\
		0&1&0&0\\
		0&0&1&0\\
		0&0&0&1
	\end{pmatrix},\;
	A_3=\begin{pmatrix}
		\xi_8^3&0&0&0\\
		0&\xi_8&0&0\\
		0&0&0&1\\
		0&0&1&0
	\end{pmatrix}.
\end{equation*}
Let $G_{48}$ (resp. $G_{96}$) be the subgroups in ${\rm PGL}(4,\C)\cong {\rm Aut}(\P^3)$ generated by  $A_1, A_2$ (resp. $A_1,A_2,A_3$). Then  $
	G_{48}\cong\SL(2,3)\rtimes C_2$, 
	$G_{96}\cong(\SL(2,3)\rtimes C_2)\rtimes C_2$. Clearly $G_{48}\subseteq {\rm Aut}(X_t)$ for any $t\in \C\setminus\{0,1\}$ and $G_{96}\subseteq {\rm Aut}(X_{-1})$.
The smooth quintic threefolds $\widehat{X}_t$ in $\P^4$ defined by $F_t+x_5^5$ admit a faithful action of $G_{48}\times C_5$. Note that $\widehat{X}_{-1}$ is isomorphic to the smooth quintic threefold in \cite[Example 2.1 (17)]{OY19}. By the classification result \cite[Theorem~2.2]{OY19} and computations similar to its proof, one can show that $\Aut(X_t)=G_{48}$ for $t\notin\{0,1,-1\}$, and $\Aut(X_{-1})=G_{96}$. We will not use these two equalities in the sequel and we omit the proof.
\end{remark}

\section{Generators of N\'eron–Severi groups of large rank}

In this section, we show that $X_{-1/8}$ has Picard number $43$, a new record for smooth quintic surfaces. Then using $5$-saturated lattices, we show that the N\'eron–Severi group ${\rm NS}(X_{-1/8})$ is generated by the classes of the lines and four conics.

Let $X=X_{-1/8}$ be the smooth quintic surface defined by $F_{-1/8}$ as in \eqref{eq:85-pencil}. Then $\sigma:=[A_1A_2]\in {\rm Aut}(X)$, where $A_1, A_2$ are as in Remark \ref{rem:aut}. Consider the smooth conic $C_1\subset\P^3$ defined by
\begin{equation}\label{eq:85-conic}
	\left\{
	\begin{aligned}
		(1-i)\sqrt[4]{24}\bigl(\sqrt[4]{12}-\sqrt3+1\bigr)x_2
		&=\bigl(\sqrt[4]{12}+\sqrt3-1\bigr)x_3+4x_4,\\[4pt]
		4i\sqrt2\bigl(x_1^2+(\sqrt[4]{108}-3)x_2^2\bigr)
		&=(\sqrt3-1)x_3^2
		+2\bigl(\sqrt3-1-\sqrt[4]{108}\bigr)x_4^2.
	\end{aligned}
	\right.
\end{equation}
Let
\begin{equation}\label{eq:85-four-conics}
	C_j:=\sigma^{j-1}(C_1),\qquad j=2,3,4.
\end{equation}
Clearly the four conics $C_j$ are on $X$. Recall that $X$ contains exactly $85$ lines $L_1,\dots,L_{85}$ .

\begin{proof}[Proof of Theorem \ref{thm:main-ns}]
By computation with the help of computer, the intersection matrix for the $89$ curves $L_1,\dots,L_{85}$, $C_1, \dots, C_4$ is of rank $43$. This implies $\rho(X)\ge 43$.

Let $G\subset {\rm Aut}(X)$ be a subgroup. Let $r$ denote the dimension of the space $H^0(X, K_X)^G$ of $G$-invariant holomorphic $2$-forms on $X$. Suppose $Y$ is a smooth projective surface birational to the quotient $X/G$. Using \cite[Proposition 5]{Shi86},  we infer the following inequality
\begin{equation}\label{eq:boundrho}
\rho(X)\le 45+2 r-b_2(Y)+\rho(Y),
\end{equation}
where $b_2(Y)$ is the second Betti number of $Y$. 

Choose $G=\langle g\rangle$, where $g=[{\rm diag}(i,-i,1,1)]$. Then $G$ is a cyclic group of order $4$. By computation of the action of $G$ on $H^0(X,K_X)\cong H^0(X,\mathcal{O}_X(1))$, we get $r=2$. The residual quartics in the planes containing the line $L_1=V(x_3,x_4)\subset \mathbb{P}^3$ give a fibration $\varphi: X\rightarrow \mathbb{P}^1$. Clearly each fiber of $\varphi$ is preserved by $G$. By direct computation, $g$ (resp. $g^2$) has no fixed points (exactly $4$ fixed points) on each smooth fiber $C$ of $\varphi$, and the quotient $C/G$ has genus one by Riemann--Hurwitz formula. Note that the smooth quintic curve $Z:=V(x_1)\cap X$ is preserved by $G$, and $Z$ meets a general fiber of $\varphi$ at $4$ points which form a single $G$-orbit. Thus $Z/G\cong \mathbb{P}^1$ gives a section of the induced fibration $\overline{\varphi}: X/G\rightarrow \mathbb{P}^1$. Then the minimal resolution $Y$ of $X/G$ is an elliptic surface with a zero section induced by $Z/G$. Converting to Weierstrass form, we infer that $Y$ is the relatively minimal elliptic surface given by 
$$y^2=x^3+12 t(t+1)(t^3+8) x^2 - 12 t^2 (t-1)^2 (t^3+8)^2 x.$$
By computation of the discriminant, the types of singular fibers are $5 I_0^*+I_4+2 I_1$ in Kodaira's notation. Then the Euler number $e(Y)$ of $Y$ is $5\cdot 6+4+2\cdot 1=36$, and $b_2(Y)=e(Y)-2=34$.

The surface $Y$ has good reduction at $p=13$. The rank of the Mordell--Weil lattice of the elliptic surface $Y\otimes \overline{\mathbb{F}}_{13}$ is at most $3$ (for example, one can verify this using the function \texttt{AnalyticInformation} of Magma (\cite{BCP})). Since the types of singular fibers for $Y\otimes \overline{\mathbb{F}}_{13}$ are the same as those for $Y$, we have $\rho(Y\otimes \overline{\mathbb{F}}_{13})\le 25 + 3=28$ by Shioda--Tate formula. Then $\rho(Y)\le 28$ since $\rho(Y)\le \rho(Y\otimes \overline{\mathbb{F}}_{13})$. Applying the inequality \eqref{eq:boundrho}, we conclude the theorem. 
\end{proof}
By checking the $d$-saturation of a sublattice in ${\rm NS}(X_{-1/8})$ orthogonal to $h$, we obtain an explicit set of generators of the group ${\rm NS}(X_{-1/8})$.
\begin{proof}[Proof of Theorem \ref{thm:main-basis}]
Recall that the line lattice $S_X$ has rank $39$. We define the lattice
$$M:=\bigl\langle S_X, [C_1],[C_2],[C_3],[C_4]\bigr\rangle\subseteq {\rm NS}(X).$$	Let $N:=M\cap h_{\rm NS}^\perp$. By computation, $\Rank (M)=43$, ${\rm disc}\, M=2^{18}\cdot3^7$, and the discriminant group
	\[
	N^\vee/N\cong
	(\Z/2\Z)^3\oplus\Z/6\Z\oplus(\Z/12\Z)^4
	\oplus\Z/24\Z\oplus\Z/120\Z.
	\]
	Then the primes satisfying
	$p^2\mid\operatorname{disc}N$ are $2$ and $3$. For each isotropic subgroup  $H\subset N^\vee/N$ with $|H|=2,3$, we confirm that
	$n_6(N_H)>n_6(N)$ using PARI/GP (\cite{Th}).
	Thus by Lemma~\ref{lem:saturation-criterion}, $N$ is $5$-saturated.
	Since $S_X\subseteq M$ and $\Rank (M)=\rho(X)$ by Theorem \ref{thm:main-ns},
	Corollary~\ref{cor:ns-generation} implies $M={\rm NS}(X)$.
\end{proof}

\begin{remark}
Some other smooth members in the pencil in Theorem \ref{thm:mainA} also have Picard number $43$ by similar arguments to the case $t=-1/8$. On the other hand, $X_{-1}$ has Picard number $39$ and its N\'eron--Severi group is generated by the classes of lines by applying Corollary~\ref{cor:ns-generation}.
\end{remark}

\begin{remark}\label{rmk:Barth}
The Barth quintic surface $X_B$ contains $75$ lines, attaining the previous record for smooth complex quintic surfaces. Rams--Sch\"utt \cite[Theorem 2.2]{RS14} proved that $X_B$ has Picard number $41$. Moreover, they derived a sublattice $M'\subseteq {\rm NS}(X_B)$ containing the line lattice $S_{X_B}$ with index $[{\rm NS}(X_B): M']=2^i$ for some $i\in \{0,\dots,4\}$ (\cite[Proposition 5.2]{RS14}). Similar to the proof of Theorem \ref{thm:main-basis}, one can verify that $M'\cap h_{{\rm NS}(X_B)}^\perp$ is $5$-saturated. Therefore, Corollary~\ref{cor:ns-generation} implies ${\rm NS}(X_B)=M'$. 
\end{remark}

\section{Further examples of surfaces with many lines}
In this section, we construct further examples of smooth surfaces with many lines in degrees $5$ and $10$. The quintic examples contain exactly $79$ lines and admit configurations of $21$ and $24$ pairwise skew lines. The example of degree $10$ contains exactly $356$ lines, including $112$ pairwise skew lines.

\subsection{Quintic surfaces with $79$ lines}
Consider the pencil $Y_t=V(P_t)\subset\P^3$, $t\in\C$, where
\begin{equation*}
	P_t=x_1^4x_2+x_1x_2^4+x_3^4x_4+x_3x_4^4+t(x_1x_2(x_3^3+x_4^3)+x_3x_4(x_1^3+x_2^3)).
\end{equation*}

\begin{theorem}\label{thm:79-members}
	The surfaces $Y_2$, $Y_{-3+\sqrt{13}}$, and $Y_{-3-\sqrt{13}}$
	are smooth and each contains exactly $79$ lines.
	Moreover, $Y_2$ contains $21$ pairwise skew lines, and each of
	$Y_{-3+\sqrt{13}}$ and $Y_{-3-\sqrt{13}}$ contains $24$ pairwise
	skew lines.
\end{theorem}

\begin{proof}
	Smoothness follows from a direct calculation of the partial derivatives.
	We count the lines using the standard Pl\"ucker stratification of the Grassmannian ${\rm Gr}(2,4)$ of lines in $\P^3$ (see e.g., \cite[Proof of Theorem~3.1]{BS07}).
	For each of the three parameters, a computation in Macaulay2 (\cite{GS})
	gives $66$ lines of the form $ax_1+bx_2+x_3=cx_1+dx_2+x_4=0$ and $10$ lines of the form $ax_1+x_2=bx_1+cx_3+x_4=0$, with parameters in $\C$.
	The remaining strata contain precisely the three coordinate lines $V(x_2,x_3)$, $V(x_1,x_4)$, and $V(x_1,x_3)$.
	Thus each surface contains exactly $66+10+3=79$ lines.
	
	We now describe the skew configurations. Let
	$D_j=\Diag(\xi_3^j,\xi_3^{-j})$; throughout, $j,k\in\{0,1,2\}$.
	For $t=2$, let $a_i$, $1\leq i\leq4$, be the roots of
	$z^4-z^3+3z^2-z+1=0$. Let $\mathcal K_2$ consist of the lines
	$$
	\binom{x_3}{x_4}=TD_j\binom{x_1}{x_2},
	\qquad
	T\in\left\{
	\begin{pmatrix}a_i&0\\0&a_i\end{pmatrix},\
	\begin{pmatrix}0&-1\\-1&0\end{pmatrix},\
	\begin{pmatrix}1&-1\\0&-1\end{pmatrix},\
	\begin{pmatrix}-1&0\\-1&1\end{pmatrix}
	\right\}.
	$$
	
	For $t=-3\pm\sqrt{13}$, let $b_1,b_2$ be the roots of
	$2b^2-(t+2)b+2=0$, and take
	$$
	U_i=\frac1{b_i+1}
	\begin{pmatrix}
		-b_i&1-b_i^2\\
		1-b_i^2&-b_i
	\end{pmatrix},
	\qquad i=1,2.
	$$
	Let $\mathcal K_t$ ($t=-3\pm\sqrt{13}$) be the set of the lines
	$$
	\binom{x_3}{x_4}=T\binom{x_i}{x_{3-i}},
	\qquad
	T\in\{b_iD_j,\ D_jU_iD_k\},
	\qquad i=1,2.
	$$
	There are $7\cdot3=21$ choices in the first construction and
	$2(3+9)=24$ in the second.
	Direct calculation verifies that the lines in each $\mathcal K_t$
	lie on $Y_t$ and are pairwise skew.
\end{proof}

\begin{remark}\label{rem:79-pencil}
The surface $Y_t$ in the pencil is smooth precisely when $t\notin\{-5,-1,1,\frac54\}$.
	Applying the preceding counting method with $t$ as a parameter
	gives, for the smooth members,
	$$
	\#\{\text{lines on }Y_t\}=
	\begin{cases}
		79,&t=2\text{ or }t=-3\pm\sqrt{13},\\
		61,&t=-\frac12,\\
		55,&t=0,\\
		43,&\text{otherwise}.
	\end{cases}
	$$
	The member $Y_0$ is the surface studied in
	\cite[Example~2.3]{Ram02}, with $55$ lines and a configuration
	of $19$ pairwise skew lines.
\end{remark}

\subsection{A surface of degree $10$ with $356$ lines}
We conclude with an example of higher degree.
\begin{theorem}\label{thm:356-lines}
	The surface $X\subset\P^3$ of degree $10$ defined by
	$$
	(x_1x_2-3x_3x_4)(x_1^8+x_2^8)
	+2(2x_3x_4-3x_1x_2)(x_3^8+2x_4^8)=0
	$$
	is smooth and contains exactly $356$ lines, including
	$112$ pairwise skew lines.
\end{theorem}
\begin{proof}
	A direct calculation of the partial derivatives verifies smoothness.
	We use the counting method described in the proof of
	Theorem~\ref{thm:79-members}.
	A computation in Macaulay2 (\cite{GS}) gives $288$ lines of the form
	$ax_1+bx_2+x_3=cx_1+dx_2+x_4=0$ and $65$ lines of the form
	$ax_1+x_2=bx_1+cx_3+x_4=0$, with parameters in $\C$.
	The remaining strata contain precisely the three coordinate lines
	$V(x_2,x_3)$, $V(x_1,x_4)$, and $V(x_1,x_3)$.
	Thus $X$ contains exactly $288+65+3=356$ lines.
	 Finally, applying the function \texttt{FindIndependentVertexSet} of Mathematica (\cite{Wo}) to the dual graph of the lines gives a set of $112$ pairwise skew lines on $X$.
\end{proof}

\end{document}